\documentclass{article}%
\usepackage{amsmath}
\usepackage{amsfonts}
\usepackage{amssymb}
\usepackage{graphicx}%
\providecommand{\U}[1]{\protect\rule{.1in}{.1in}}
\usepackage{amsfonts, amsmath, amssymb, amsthm, color}
\usepackage[hmargin=2.75cm, vmargin=2.75cm]{geometry}
\usepackage{txfonts}
\usepackage{graphicx, setspace}
\newcommand{\R}{{{\mathbb R}}}
\newtheorem{theorem}{Theorem}

\newtheorem{corollary}[theorem]{Corollary}

\newtheorem{definition}[theorem]{Definition}

\newtheorem{lemma}[theorem]{Lemma}

\newtheorem{remark}[theorem]{Remark}

\begin{document}

\title{On Helmholtz equation and Dancer's type entire solutions for nonlinear elliptic  equations }

\author{Yong Liu, Juncheng Wei}

\newcommand{\Addresses}{{\bigskip \footnotesize

\medskip

\noindent {Yong Liu} \textsc{School of Mathematics and Physics, North China Electric Power University, Beijing, China,} \par\nopagebreak
noindent \textit{E-mail address:  liuyong@ncepu.edu.cn}

\medskip
\noindent (Juncheng Wei) \textsc{Department of Mathematics, University of British Columbia, Vancouver, B.C., Canada, V6T 1Z2}\par\nopagebreak
\noindent \textit{E-mail address}: \texttt{jcwei@math.ubc.ca}}}

\date{\today}
\maketitle

\begin{abstract}
Starting from a bound state (positive or sign-changing) solution to
$$ -\Delta \omega_m =|\omega_m|^{p-1} \omega_m -\omega_m \ \ \mbox{in}\  \R^n, \ \omega_m \in H^2 (\R^n)$$
and  solutions to the Helmholtz equation
$$ \Delta u_0 + \lambda u_0=0 \ \ \mbox{in} \ \R^n, \ \lambda>0, $$
we build new Dancer's type entire solutions to the nonlinear scalar equation
$$ -\Delta u =|u|^{p-1} u-u \ \ \mbox{in} \ \R^{m+n}. $$

\end{abstract}

\maketitle

\section{\bigskip Introduction}

The  purpose of this  note is to construct new entire (positive or sign-changing) solutions for the
elliptic equation
\begin{equation}
-\Delta u =\left\vert u \right\vert^{p-1}u -u\ \ \text{ in }\ \R^{m+n},\label{Sch}%
\end{equation}
where the exponent $p>1$  and satisfies   some conditions. A solution to (\ref{Sch}) corresponds to a standing wave  to the nonlinear Schr\"{o}dinger equation
\begin{equation}
-iu_t = \Delta u + |u|^{p-1} u, \ \ \mbox{in} \ \R^{m+n}.
\end{equation}
 It  also serves as models in different areas of applied mathematics such as pattern formation in mathematical biology.

  Equation (\ref{Sch})  has been studied extensively and there is a vast literature on this
subject. Let us first describe  the classical example of Dancer's solutions. To describe these solutions, we recall that when $n=0,$ equation (\ref{Sch}) reduces to
\begin{equation}
-\Delta u=\left\vert u\right\vert ^{p-1}u-u\text{ in }\ \R^{m}, \label{ODE}%
\end{equation}
which admits a unique positive radially symmetric ground state solution $w_m$ decaying exponentially fast to zero at infinity, provided that $p$ is subcritical, i.e. $ 1<p<\frac{m+2}{m-2}$.
The linearized equation of $\left(  \ref{ODE}\right)  $ around $w_m$ is
\[
L_{w_m}\eta=-\Delta \eta +\left(  1-pw_m^{p-1}\right)  \eta,
\]
acting on $H^{2}\left(  \mathbb{R}^{m}\right).$ The essential spectrum of $L_{w_m}$ is
$[1,+\infty).$ It is known that this operator has a unique negative eigenvalue $-\lambda_{1}.$ We choose a
corresponding positive eigenfunction $Z_1\left(  x\right)  $.
Dancer \cite{Dancer} first constructed new positive solutions with only partial decaying using bifurcation theory.
He proved that for $n=1$, there exists a family of solutions to (\ref{Sch}) with the following behavior:
\begin{equation}
\label{new12}
 u(x, y)= w_m(x)+ \epsilon Z_1(x) \cos (\sqrt{\lambda_1} y) + o(\epsilon e^{-\frac{1}{2}|x|}), \ |\epsilon|<<1, \ (x,y) \in \R^{m+1}.
 \end{equation}
 Moreover such solutions are periodic in $y$.   We call this type of solutions as Dancer's type.

The existence of Dancer's solution generates lots of interest in constructing other type of solutions. Variational and gluing methods have been
successfully applied in the construction of solutions. There is  also a deep connection between the solutions of (\ref{Sch}) and the constant mean curvature surfaces (CMC). We refer to
\cite{AMPW, Dancer, Mal, Musso,Santra} and the references therein for more discussions.

In this paper we extend Dancer's type solutions to general dimensions $n\geq 3$. Let $\omega_m$ be a  bound state solution (not necessary positive) of
equation $\left(  \ref{ODE}\right)  $ with $\omega_m \left(  x\right)
\rightarrow0$ as $\left\vert x\right\vert \rightarrow+\infty.$ Existence
results of this type solutions for subcritical exponent $p$ could be found in
\cite{B1,B2}. There are also plenty of sign-changing radial solutions. In fact, for each integer $k\geq 0$ there are sign-changing radial solutions with $k$ zeroes.  There are also infinitely many sign-changing solutions without any symmetry. See \cite{BW, AMPW, conti}.  Slightly abusing the notation, we use $\left(  x,y\right)  $ to
denote the vectors in $\mathbb{R}^{m+n}$, where $x$ represents the first $m$
coordinates. Denote the positive eigenvalues of $-L_{\omega_m }$ as
$ \lambda_1,...,\lambda_k$, with the corresponding eigenfunctions $Z_1, ..., Z_k$. Let $Z_{k+1}, ..., Z_l$ be the eigenfunctions of the zero eigenvalue. We could also assume that $Z_1,...,Z_l$ consists an orthonormal basis for the nonnegative eigenspace in $L^2(\mathbb{R}^m)$. (Note that we don't assume the non-degeneracy of $w$.)

Our main result can be stated roughly as follows: Starting from solutions to the Helmholtz equation
\begin{equation}
\label{Helm11}
\Delta u_j +\lambda_j u_j=0 \ \ \mbox{in} \ \R^n, j=1,..., k,
\end{equation}
we find a family of solutions to (\ref{Sch}) with the following asymptotic behavior:
\begin{equation}
\label{new13}
 u(x, y)= \omega_m (x)+ \epsilon \sum_{j=1}^k Z_j (x) u_j (y) + o(  \epsilon e^{-\frac{1}{2}|x|}), \ \ (x,y) \in \R^{m+n}.
 \end{equation}
 As a consequence, for $n\geq3,$ there are abundance of solutions near $\omega_m (x)$. Unlike the classical Dancer's solution, these solutions are not
periodic in the $y$ variable. As a matter of fact, they converge to $\omega_m \left(
x\right)  $ as $|y|\rightarrow\infty.$  The existence of these type of solutions makes it more difficult to classify entire solutions near the ground state profile. Nevertheless we believe that all solutions near $\omega_m$ should be described by (\ref{new13}). We refer to \cite{BF, GMX} and the references therein for partial progress on this issue.

Our idea of the proof is in the spirit
similar to that of \cite{Susana}, where existence of small amplitude solutions to the
Ginzburg-Landau equation in dimension 3 and 4 has been proved.
 In \cite{Weth}, using dual variational method,
the existence of a sequence of solutions $u_{k}$ of
$$\Delta u + u + |u|^{p-1} u=0\ \ \ \mbox{in} \ \R^{n}, $$
 with $\left\Vert u_{k}\right\Vert _{L^{p}}\rightarrow+\infty,$ has been
proved under certain condition on $p$ and $n.$ A similar functional-analytical frame was used.

To describe our main results, we need to introduce some notations. Let $\lambda>0$ be a fixed positive number. Consider the so-called Helmholtz equation
\begin{equation}
\Delta u+\lambda u=0\text{ in }\mathbb{R}^{n}. \label{Helm}%
\end{equation}
We are interested in solutions to (\ref{Helm}) with the following decaying property
\begin{equation}
\left\vert u \left(  y\right)  \right\vert \leq C\left(  1+\left\vert
y\right\vert \right)  ^{-\frac{n-1}{2}}. \label{upperbound}%
\end{equation}

There are plenty  many solutions to (\ref{Helm}) satisfying (\ref{upperbound}).  We start with radial solutions.
Let $s$ be a
parameter and $n\geq2$. Consider the equation
\begin{equation}
\varphi^{\prime\prime}+\frac{n-1}{r}\varphi^{\prime}+\left(1
-\frac{s^{2}}{r^{2}}\right)  \varphi=0. \label{Bs}%
\end{equation}
For $n=2,$ it has a regular solution $J_{2,s}$, which is the Bessel functions of the first kind. For general $n\geq2,$ $\left(
\ref{Bs}\right)  $ has a regular solution
\[
J_{n,s}\left(  r\right)  =r^{1-\frac{n}{2}}J_{2,\sqrt{\left(  \frac{n}%
{2}-1\right)  ^{2}+s^{2}}}\left(  r\right)  .
\]
It is known that
\[
J_{n,s}\left(  r\right)  \leq C\left(  1+r\right)  ^{-\frac{n-1}{2}}.
\]
Hence $J_{n, 0} (\sqrt{\lambda} r)$ is a solution to (\ref{Helm})-(\ref{upperbound}).

Given finitely many points $y_{j}\in \mathbb{R}^{n},j=1,...,q,$ functions of the form
\[%
{\displaystyle\sum\limits_{j=1}^{q}}
J_{n,0}\left( \sqrt{\lambda}( y-y_{j})\right)
\]
are also  solutions of $\left(  \ref{Helm}\right)  $ satisfying $\left(
\ref{upperbound}\right).$  In \cite{Enciso}, solutions of (\ref{Helm}) satisfying (\ref{upperbound})  with zero level sets arbitrarily close to any compact surfaces are constructed.

 Our main result states that from these solutions  of the Helmholtz equation we could
construct solutions of $\left(  \ref{Sch}\right).$

\begin{theorem}
\label{main0}Let $n\geq6$ and $p\geq\frac{n+2}{n-2}.$ For any $\varepsilon$
with $\left\vert \varepsilon\right\vert $ small enough, there is a solution
$u_{\varepsilon}$ to the equation
\[
\Delta u+\left\vert u\right\vert ^{p-1}u-u=0,\text{ in }\mathbb{R}^{m+n},
\]
such that
\[
u_{\varepsilon}=\omega_{n}\left(  x\right)  +\varepsilon  \sum_{j=1}^k Z_{j}\left(
x\right)  u_{j}\left(  y\right)  +o\left(  \varepsilon\right)
\]
where $ u_j$ are solutions of (\ref{Helm})-(\ref{upperbound}), with $\lambda_j$ being the negative eigenvalues of $ L_{\omega_m}$.

\end{theorem}

As a corollary of the proof of this theorem, in the case that $\omega_m$ is the positive radially symmetric solution $w_m$, we have the following result.

\begin{corollary}
\label{main}Let $n\geq5$ and $p\geq\frac{n+3}{n-1}.$ Let $w_m$ be the unique positive solutions of (\ref{Sch}).  For any $\varepsilon$
with $\left\vert \varepsilon\right\vert $ small enough, there is a positive
solution $u_{\varepsilon}$ to $\left(  \ref{Sch}\right)  $ such that
\[
u_{\varepsilon}=w_m\left(  x\right)  +\varepsilon Z_1\left(  x\right)
u_{1}\left(  y\right)  +o\left(  \varepsilon\right)  .
\]

\end{corollary}
Observe that in this corollary, we allow $n=5$. This is partly due to the fact that $w_m$ is nondegenerate in certain sense.

\begin{remark}
We don't know the decay rates of these solutions to $\omega_m$ as $y$ goes to infinity.
\end{remark}
When we are considering the existence of solutions radially symmetric in the
$y$ variable, the requirement that $p\geq\frac{n+3}{n-1}$ could be slightly
relaxed. This is the content of our next theorem.

\begin{theorem}
\label{main1}Let $n\geq$ $4$ and $p>\frac{n+1}{n-1}.$ For any $\varepsilon$
with $\left\vert \varepsilon\right\vert $ small enough, there is a positive
solution $u_{\varepsilon}$ to $\left(  \ref{Sch}\right)  $ which is radially
symmetric in the $y$ variable and
\[
u_{\varepsilon}\left(  x,y\right)  =w_m\left(  x\right)  +\varepsilon
J_{n,0}\left( \sqrt{\lambda_1} \left\vert y\right\vert \right)  +o\left(  \varepsilon\right)
.
\]

\end{theorem}

When $p$ is an integer, using the same method, we could obtain similar result for $n=$ $3.$

\begin{theorem}
\label{main1'}Let $n=$ $3$ and $p>1$ be an integer. For any $\varepsilon$ with
$\left\vert \varepsilon\right\vert $ small enough, there is a positive
solution $u_{\varepsilon}$ to $\left(  \ref{Sch}\right)  ,$ radially symmetric
in the $y,$ such that
\[
u_{\varepsilon}\left(  x,y\right)  =w_m\left(  x\right)  +\varepsilon
J_{n,0}\left( \sqrt{\lambda_1} \left\vert y\right\vert \right)  +o\left(  \varepsilon\right)
.
\]

\end{theorem}
\begin{remark}
An open question is  the case of $n=2$. We don't know whether or not there are similar solutions. We believe that our method of proof for these theorems could potentially be used in other settings.
\end{remark}

We will use contraction mapping principle to prove these results.  The conditions on $p$ and $n$ are used to ensure the contraction mapping properties.  In Section
2, we prove Theorem \ref{main'} and sketch the proof of Corollary \ref{main}. In
section 3, we prove Theorem \ref{main1} and Theorem \ref{main1'}.

Throughout the paper, we use $C$ to denote a general constant which may vary
from step to step.

\medskip

\textit{Acknowledgement.}
The research of J. Wei is partially supported by NSERC of Canada.
Part of the paper was finished while Y. Liu was visiting the University of British Columbia in 2016.
He appreciates the institution for its hospitality and financial support.

\section{Proof of Theorem \ref{main0} and Theorem \ref{main}}
We first prove Theorem \ref{main0}. For each $\varepsilon$ with $\left\vert \varepsilon\right\vert $ small enough,
we will construct a solution $u_{\varepsilon}$ to $\left(  \ref{Sch}%
\right)  $ in the form:
\[
u_{\varepsilon}\left(  x,y\right)  =\omega_m\left(  x\right)  +\sum_{i=1}^{l}Z_i\left(  x\right)
f_i\left(  y\right)  +\Phi\left(  x,y\right)  ,
\]
where we require that $\Phi$ is
orthogonal to $Z_i$ in the following sense:
\[
\int_{\mathbb{R}^m}\Phi\left(  x,y\right)  Z_i\left(  x\right)  dx=0,\text{ for any }y,i.
\]
We obtain that the equations satisfied by $f_i$ and $\Phi$ are%
\[
-\Delta\Phi+\left(  1-p\left\vert \omega_{m}\right\vert ^{p-1}\right)  \Phi+\left[  -\Delta f_i-\lambda
_{i}f_i\right]  Z_i\left(  x\right)  =N\left(  f,\Phi\right)  .
\]
Here\[
N\left(  f,\Phi\right)  =\left\vert u_{\varepsilon}\right\vert ^{p-1}%
u_{\varepsilon}-\left\vert \omega_{m}\right\vert ^{p-1}\omega_{m}-p\left\vert
\omega_{m}\right\vert ^{p-1}\left(
{\displaystyle\sum\limits_{i=1}^{l}}
Z_{i}\left(  x\right)  f_{i}\left(  y\right)  +\Phi\right)  .
\]

Introduce the constant $\bar{p}:=\min\left\{  p,2\right\}  .$ We have the
following estimate for $N:$%
\begin{equation}\label{N}
\left\vert N\left(  f,\Phi\right)  \right\vert \leq C\left(
{\displaystyle\sum\limits_{i=1}^{l}}
\left\vert f_{i}\right\vert ^{\bar{p}}+\left\vert \Phi\right\vert ^{\bar{p},
}\right)
\end{equation}
provided that $\left\vert f_i\right\vert $ and $\left\vert \Phi\right\vert $ are small.

Let $f_i\left(  y\right)  =\varepsilon u_{i}\left(  \left\vert y\right\vert
\right)  +h_i\left(  y\right)  $, where for $i=1,..., k$, $\lambda_i>0$ and $u_i$ satisfies
\begin{equation}
\Delta u_i +\lambda_i u_i=0 \ \mbox{in} \ \R^n,  \ |u_i|\leq C (1+|y|)^{-\frac{n-1}{2}}
\end{equation}
For $i=k+1,..., l$, $ \lambda_i=0$ and $u_i \equiv 0$.

 Then we need to solve%
\begin{equation}
-\Delta\Phi+\left(  1-p\left\vert \omega_{m}\right\vert ^{p-1}\right)  \Phi+\left[  -\Delta h_i-\lambda
_{i}h\right]  Z_i\left(  x\right)  =N\left(  f,\Phi\right)  .
\label{eq1}%
\end{equation}
For the sake of convenience, we introduce the notation
\[
L\Phi:=-\Delta\Phi+\left(  1-p\left\vert \omega_{m}\right\vert ^{p-1}\right)  \Phi.
\]
To get a solution for $\left(  \ref{eq1}\right)  ,$ it will be sufficient to deal
with the system
\begin{equation}
\left\{
\begin{array}
[c]{l}
-\Delta h_i-\lambda_{i}h_i=\int_{\mathbb{R}^m}\left[  Z_i\left(  x\right)  N\left( f,\Phi\right)  \right]  dx, i=1,...,l,\\
L\Phi=N\left(f,\Phi\right)  -\sum_{i=1}^{l}Z_i\left(  x\right)  \int%
_{\mathbb{R}^m}\left[  Z_i\left(  x\right)  N\left(  f,\Phi\right)
\right]  dx.
\end{array}
\right.  \label{sys}%
\end{equation}
Throughout the paper, we use $q^{\prime}$ to denote the conjugate exponent of
$q.$ That is,
\[
q^{\prime}=\frac{q}{q-1}.
\]
We need the following  important generalized Sobolev type inequality (Theorem 2.3 in
\cite{KRS}).

\begin{lemma}
\label{Sobolev}Suppose the exponent $q$ satisfies
\begin{equation}
\frac{2}{n+1}\leq\frac{1}{q^{\prime}}-\frac{1}{q}\leq\frac{2}{n}. \label{qq'}%
\end{equation}
Then
\begin{equation}
\left\Vert \Phi\right\Vert _{L^{q}\left(  R^{n}\right)  }\leq C\left\Vert
\Delta\Phi+\Phi\right\Vert _{L^{q^{\prime}}\left(  R^{n}\right)  }.
\label{ine}%
\end{equation}

\end{lemma}

Note that if $u\in L^{q},$ then $u^{q-1}\in L^{q^{\prime}}.$ Inequality
$\left(  \ref{qq'}\right)  $ is equivalent to
\[
\frac{n-2}{n}\leq\frac{2}{q}\leq\frac{n-1}{n+1}.
\]
That is,
\[
\frac{2\left(  n+1\right)  }{n-1}\leq q\leq\frac{2n}{n-2}.
\]
Recall that $\bar{p}=\min\left\{  p,2\right\}  .$ Hence under the assumption
that $n\geq 6$ and $p\geq \frac{n+2}{n-2}$, we have
\[
\bar{p}\geq\frac{2n  }{n-2}-1=\frac{n+2}{n-2}.
\]

To solve $\left(  \ref{sys}\right)  ,$we first consider the solvability of the
equation
\begin{equation}
L\Phi=\xi\label{linear1}%
\end{equation}
for given function $\xi.$ For this purpose, we introduce the functional space
to work with.

\begin{definition}
The space $E_{\alpha}$ consists of functions $\xi$ defined on $\mathbb{R}^{m+n}$
satisfying
\[
\left\Vert \xi\right\Vert _{\ast,\alpha}:=\left\Vert \xi\right\Vert
_{L^{\alpha}\left(  \mathbb{R}^{m+n}\right)  }+\left\Vert \xi\right\Vert _{L^{\infty
}\left(  \mathbb{R}^{m+n}\right)  }<+\infty.
\]
The space $\bar{E}_{\alpha}$ consists of $l$-tuple of functions $\eta=(\eta_1,...,\eta_l)$ with $\eta_i$ defined on $\mathbb{R}^{n}$,
satisfying
\[
\left\Vert \eta\right\Vert _{\ast\ast,\alpha}:=\Sigma_{i=1}^{l}\left\Vert \eta_i\right\Vert
_{L^{\alpha}\left(  \mathbb{R}^{n}\right)  }+\Sigma_{i=1}^{l}\left\Vert \eta_i\right\Vert _{L^{\infty
}\left(  \mathbb{R}^{n}\right)  }<+\infty.
\]

\end{definition}

We will choose $q$ to be $\frac{2n  }{n-2}.$ Then
\[
q^{\prime}=\frac{q}{q-1}=\frac{2n  }{n+2}.
\]

\begin{lemma}
\label{Linearsol}Suppose $\left\Vert
\xi\right\Vert _{\ast,q^{\prime}}<+\infty$ and
\[
\int_{\mathbb{R}^m}\xi\left(  x,y\right)  Z_i\left(  x\right)  dx=0,\text{ for any }y, i.
\]
Then the equation $\left(  \ref{linear1}\right)  $ has a solution $\Phi$
satisfying
\[
\left\Vert \Phi\right\Vert _{\ast,q}\leq C\left\Vert \xi\right\Vert
_{\ast,q^{\prime}}%
\]
and
\[
\int_{\mathbb{R}^m}\Phi\left(  x,y\right)  Z_i\left(  x\right)  dx=0,\text{ for any }y, i.
\]

\end{lemma}

\begin{proof}
Note that $q=\frac{2n }{n-2}>2$ and hence $q^{\prime}<2.$ We
then have%
\[
\left\Vert \xi\right\Vert _{L^{2}}\leq C\left\Vert \xi\right\Vert
_{\ast,q^{\prime}}.
\]
Consider the operator $L$ acting on the space of functions in $H^{2}\left(
\mathbb{R}^{m+n}\right)  $ even in $x$ variable and which additionally orthogonal to
$Z_i\left(  x\right)  $ for each $y, i.$ Then $0$ is not in the spectrum of $L$ and
hence we have an even solution $\Phi$ for the equation $L\Phi=\xi,$ with
$\left\Vert \Phi\right\Vert _{L^{2}}\leq C\left\Vert \xi\right\Vert _{L^{2}%
}\leq C\left\Vert \xi\right\Vert _{\ast,q^{\prime}}$ and
\[
\int_{\mathbb{R}^m}\Phi\left(  x,y\right)  Z_i\left(  x\right)  dx=0,\text{ for any }y, i.
\]
On the other hand, since we impose the orthogonality condition on $\Phi,$ we
also have the $L^{\infty}$ bounds for $\Phi,$ that is,
\[
\left\Vert \Phi\right\Vert _{L^{\infty}}\leq C\left\Vert \xi\right\Vert
_{L^{\infty}}\leq C\left\Vert \xi\right\Vert _{\ast,q^{\prime}}.
\]
Therefore, using the fact that $q>2,$ we find that $\left\Vert \Phi\right\Vert
_{L^{q}}\leq C\left\Vert \Phi\right\Vert _{\ast,2}\leq C\left\Vert
\xi\right\Vert _{\ast,q^{\prime}}.$ This finishes the proof.
\end{proof}

With all these estimates at hand, we could use the contraction mapping
principle to prove Theorem \ref{main0}.

\begin{proof}
[Proof of Theorem \ref{main0}]For each $h=(h_1,...,h_l)$ with $\left\Vert h\right\Vert _{\ast\ast
,q}\leq M_{1}\varepsilon^{\bar{p}}$, where $M_{1}$ is a large constant, we
consider the equation
\[
L\Phi=N\left( f,\Phi\right)  -\sum_{i=1}^{l}Z_i\left(  x\right)  \int%
_{\mathbb{R}^m}\left[  Z_i\left(  x\right)  N\left( f,\Phi\right)
\right]  dx.
\]
Note that the right hand side of this equation is automatically orthogonal to
$Z\left(  x\right)  $ for each $y.$ By Lemma \ref{Linearsol}, we could write
it as
\[
\Phi=L^{-1}\left[  N\left( f,\Phi\right)  -\sum_{i=1}^{l}Z_i(x)\int_{\mathbb{R}^m}Z_i\left(
x\right)  N\left( f,\Phi\right)  dx\right]  :=\bar
{N}\left(  h,\Phi\right)  .
\]
 Observe that the function $\left\vert u_{i}\right\vert ^{\bar{p}}$ is in
$L^{q^{\prime}}:$
\[
\int_{\mathbb{R}^{n}}\left\vert u_{i}\right\vert ^{\bar{p}q^{\prime}}dy\leq C\int%
_{0}^{+\infty}\left(  1+r\right)  ^{-\frac{n-1}{2}\frac{n+2}{n-2}%
\frac{2n  }{n+2}}r^{n-1}dr\leq C.
\]

Now suppose $\left\Vert \Phi\right\Vert _{\ast,q}\leq M_{2}\varepsilon
^{\bar{p}},$ where $M_{2}$ is a large constant. Then using the fact that
$\bar{p}\geq q-1,$ we deduce
\begin{align*}
\left\Vert \left\vert \Phi\right\vert ^{\bar{p}}\right\Vert _{L^{q^{\prime}%
}\left(  \mathbb{R}^{m+n}\right)  }  &  =\left(  \int\left\vert \Phi\right\vert
^{\bar{p}q^{\prime}}\right)  ^{\frac{1}{q^{\prime}}}\leq\left(  \left\vert
\Phi\right\vert ^{\left(  \bar{p}-\left(  q-1\right)  \right)  q^{\prime}}%
\int\left\vert \Phi\right\vert ^{\left(  q-1\right)  q^{\prime}}\right)
^{\frac{1}{q^{\prime}}}\\
&  \leq M_{2}\varepsilon^{\bar{p}^{2}}.
\end{align*}
Also we observe that
\begin{align*}
\int_{\mathbb{R}^{n}}\left(  \int_{\mathbb{R}^m}Z_i\left(  x\right)  \left\vert \eta\left(
x,y\right)  \right\vert dx\right)  ^{q^{\prime}}dy  &  \leq\int_{\mathbb{R}^{n}}\left[
\left(  \int_{\mathbb{R}^m}\left[  Z\left(  x\right)  \right]  ^{q}dx\right)  ^{\frac
{1}{q}}\left(  \int_{\mathbb{R}^m}\left\vert \eta\left(  x,y\right)  \right\vert
^{q^{\prime}}dx\right)  ^{\frac{1}{q^{\prime}}}\right]  ^{q^{\prime}}dy\\
&  \leq C\left\Vert \eta\right\Vert _{L^{q^{\prime}}}^{q^{\prime}}.
\end{align*}
This together with $\left(  \ref{N}\right)  $ implies that
\[
\left\Vert N\left( f,\Phi\right)  \right\Vert
_{L^{q^{\prime}}}\leq C\varepsilon^{\bar{p}}+\left(  M_{1}+M_{2}\right)
\varepsilon^{\bar{p}^{2}}.
\]
On the other hand, it follows from direct computation that
\[
\left\Vert N\left( f,\Phi\right)  \right\Vert _{L^{\infty}%
}\leq C\varepsilon^{\bar{p}}+\left(  M_{1}+M_{2}\right)  \varepsilon^{\bar
{p}^{2}}.
\]
We could then check that $\bar{N}$ maps the balls of radius $M_{2}%
\varepsilon^{\bar{p}}$ of $E_{q}$ into itself for $M_{2}$ large enough and
$\varepsilon$ sufficiently small.

Next we show that $\bar{N}$ is a contraction mapping. To see this, for two
function $\Phi_{1},\Phi_{2},$ we compute
\begin{align*}
\left\Vert N\left(  h,\Phi_{1}\right)  -N\left(  h,\Phi_{2}\right)
\right\Vert _{L^{q^{\prime}}}  &  \leq C\varepsilon^{\bar{p}-1}\left\Vert
\Phi_{1}-\Phi_{2}\right\Vert _{L^{q^{\prime}}},\\
\left\Vert N\left(  h,\Phi_{1}\right)  -N\left(  h,\Phi_{2}\right)
\right\Vert _{L^{\infty}}  &  \leq C\varepsilon^{\bar{p}-1}\left\Vert \Phi
_{1}-\Phi_{2}\right\Vert _{L^{\infty}}.
\end{align*}
This in turn will imply that $\bar{N}$ is a contraction map. We then conclude
that it has a unique fixed point in the ball of radius $M_{2}\varepsilon
^{\bar{p}}$, denote it by $\Phi_{h}.$

To solve the system $\left(  \ref{sys}\right)  $, it remains to solve the system of
equations
\[
-\Delta h_j-\lambda_{j}h_j=\int_{\mathbb{R}^m}\left[  Z_j\left(  x\right)  N\left( f,\Phi_{h}\right)  \right]  dx, j=1,...,l.
\]
We write it as
\[
h_j=\mathfrak{D}_j\left(  \int_{\mathbb{R}^m}Z_j\left(  x\right)  N\left( f,\Phi_{h}\right)  dx\right)  :=D_j\left(  h\right), j=1,...,l .
\]
Here the operator $\mathfrak{D}_j=\lim_{\varepsilon\rightarrow0}\left(
-\Delta-\lambda_{j}+i\varepsilon\right)  ^{-1}.$  Then we are finally lead to solve the fixed point problem
\begin{equation}
h=D(h):=(D_1(h),...,D_l(h)).
\end{equation}
Making use of the estimate of Lemma \ref{Sobolev}, we could show that
$D\left(  h\right)  $ maps the ball of radius $M_{1}%
\varepsilon^{\bar{p}}$ of $\bar{E}_{q}$ into itself for $M_{1}$ large and
$\varepsilon$ small and is a contraction mapping. We therefore get a fixed
point $h$ for this mapping, which yields a solution for our original problem.
\end{proof}

In the rest of this section, we sketch the proof of Corollary \ref{main}. In this case, we seek for a solution in the form
\[
u_{\varepsilon}\left(  x,y\right)  =w_{m}\left(  x\right)  +Z_{1}\left(
x\right)  f\left(  y\right)  +\Phi\left(  x,y\right)  .
\]
Note that $Z_1$ is radially symmetric in $x$. Here $\Phi$ is required to be radially symmetric in $x$ and orthogonal to $Z_1$ for each $y$. We write the function $f$ in the form  $f\left(  y\right)  =\varepsilon u_{1}\left( y\right)  +h\left(  y\right).$ Then similar as before, we prove the existence of solution by contraction mapping principle in suitable Banach spaces. Here we could work with functions $h$ in the space $F_\alpha$, and $\Phi$ in the space $\bar{F}_\alpha$, defined similarly as before.
\begin{definition}
The space $F_{\alpha}$ consists of functions $\xi$ defined on $\mathbb{R}^{m+n}$ which is radially symmetric in $x$,
satisfying
\[
\left\Vert \xi\right\Vert _{\ast,\alpha}:=\left\Vert \xi\right\Vert
_{L^{\alpha}\left(  \mathbb{R}^{m+n}\right)  }+\left\Vert \xi\right\Vert _{L^{\infty
}\left(  \mathbb{R}^{m+n}\right)  }<+\infty.
\]
The space $\bar{F}_{\alpha}$ consists of functions $\eta$ defined on $\mathbb{R}^{n}$,
satisfying
\[
\left\Vert \eta\right\Vert _{\ast\ast,\alpha}:=\left\Vert \eta\right\Vert
_{L^{\alpha}\left(  \mathbb{R}^{n}\right)  }+\left\Vert \eta\right\Vert _{L^{\infty
}\left(  \mathbb{R}^{n}\right)  }<+\infty.
\]

\end{definition}
Then we choose the exponent $\alpha$ to be \[\alpha=\frac{2(n+1)}{n-1}.\]
With this choice, under the assumption of Corollary \ref{main}, for $n\geq 5$, one has $\bar{p}=\min\{p,2\}\geq \alpha-1$. Then the rest of the proof follows from the arguments in that of Theorem \ref{main'}.

\section{Proof of Theorem \ref{main1} and Theorem \ref{main1'}}

We first prove Theorem \ref{main1}. The idea is still using contraction
mapping principle but we make the advantage of radial symmetry.  For the sake of convenience, we drop the subscript of $w_m$ and simply write it as $w$. We also write $Z_1$ as $Z$, $\lambda_1$ as $\lambda$. Denoting $r=\left\vert y\right\vert .$ We are looking for a
solution $u_{\varepsilon}$ in the form
\[
u_{\varepsilon}\left(  x,y\right)  =w\left(  x\right)  +Z\left(  x\right)
f\left(  r\right)  +\Phi.
\]
Here $\Phi$ is radially symmetric in $x$ and in $y.$ Plug this into the
equation
\[
-\Delta u_{\varepsilon}+u_{\varepsilon}-
\left\vert u_{\varepsilon}\right\vert ^{p}%
=0,
\]
we get%
\[
L\Phi+\left[  -f^{\prime\prime}-\frac{n-1}{r}f^{\prime}-\lambda f\right]
Z\left(  x\right)  =N\left(  f,\Phi\right).
\]
Let $f\left(  r\right)  =\varepsilon J_{n,0}\left(  r\right)  +h\left(
r\right)  .$ Then we need to solve%
\begin{equation}
L\Phi+\left[  -h^{\prime\prime}-\frac{n-1}{r}h^{\prime}-\lambda h\right]
Z\left(  x\right)  =N\left(  \varepsilon J_{n,0}+h,\Phi\right)  . \label{fi}%
\end{equation}
Here we require that
\[
\int_{\mathbb{R}^m}\Phi\left(  x,y\right)  Z\left(  x\right)  dx=0,\text{ for any }y.
\]
We are lead to solve
\[
\left\{
\begin{array}
[c]{l}%
-h^{\prime\prime}-\frac{n-1}{r}h^{\prime}-\lambda h=\int_{\mathbb{R}^m}Z\left(
x\right)  N\left(  \varepsilon J_{n,0}+h,\Phi\right)  dx\\
L\Phi=N\left(  \varepsilon J_{n,0}+h,\Phi\right)  -Z\left(  x\right)  \int%
_{\mathbb{R}^m}Z\left(  x\right)  N\left(  \varepsilon J_{n,0}+h,\Phi\right)  dx.
\end{array}
\right.
\]

We need to solve the equation
\begin{equation}
L\Phi=\xi\label{linear}%
\end{equation}
for given function $\xi$ defined in $\mathbb{R}^{m+n}.$

\begin{definition}
The space $S_{\beta}$ consists of those functions $\xi$ radially symmetric in the $x$ and $y$ variable, with
\[
\left\Vert \xi\right\Vert :=\sup_{\left(  x,y\right)  }\left\Vert \xi\left(
x,y \right)  \left(  1+|y|\right)  ^{\beta}e^{-\frac{1}{2}\left\vert x\right\vert
}\right\Vert _{C^{0,\alpha}\left(  B_{\left(  x,y\right)  }\left(  1\right)
\right)  }<+\infty.
\]
The space $\bar{S}_{\beta}$ consists of those radially symmetric functions $\eta $ defined in $\mathbb{R}^{n}$,  with
\[
\left\Vert \eta\right\Vert _{\symbol{94},\beta}:=\sup\left\Vert \left(
1+|y|\right)  ^{\beta}\eta\left(  y\right)  \right\Vert _{C^{0,\alpha}\left(
B_{y}\left(  1\right)  \right)  }<+\infty.%
\]

\end{definition}

\begin{lemma}
\label{S}Suppose $\xi$ is function radially symmetric in $x,y$
with $\left\Vert \xi\right\Vert <+\infty$ and
\[
\int_{\mathbb{R}^m}\xi\left(  x,y\right)  Z\left(  x\right)  dx=0,\text{ for any }y.
\]
Then the equation $\left(  \ref{linear}\right)  $ has a solution $\Phi,$
radially symmetric in $x,y$, satisfying
\[
\left\Vert \Phi\right\Vert \leq C\left\Vert \xi\right\Vert .
\]

\end{lemma}

\begin{proof}
The proof of this type result is by now standard, we omit the details and
refer to \cite{M} for a proof.
\end{proof}

Next we proceed to the analysis of the first equation of the system. Let
$\eta$ be a function which decays at certain rate at infinity. The
homogeneous equation
\[
h^{\prime\prime}+\frac{n-1}{r}h^{\prime}+\lambda h=0
\]
has two linearly independent solutions $J_{n}\left(  \cdot\right)  $ and
$N_{n}\left(  \cdot\right)  .$ They both decay like $r^{-\frac{n-1}{2}}$ at
infinity. $J_{n}$ is regular near $0$ and $N_{n}$ is singular near $0.$
Moreover, $N_{n}\left(  r\right)  =O\left(  r^{2-n}\right)  $ for $r$ close to
zero. Variation of parameter formula tells us that the function
\[
N_{n}\left(  r\right)  \int_{0}^{r}J_{n}\left(  s\right)  \eta\left(
s\right)  s^{n-1}ds-J_{n}\left(  r\right)  \int_{0}^{r}N_{n}\left(  s\right)
\eta\left(  s\right)  s^{n-1}ds
\]
is a solution of the nonhomogeneous equation
\begin{equation}
h^{\prime\prime}+\frac{n-1}{r}h^{\prime}+\lambda h=\eta. \label{OD}%
\end{equation}

\begin{lemma}
Let $\eta\in\bar{S}_{\beta}$ with $\beta>\frac{n+1}{2}.$ Then the equation
$\left(  \ref{OD}\right)  $ has a solution $h\in\bar{S}_{\frac{n-1}{2}}$
satisfies%
\[
\left\Vert h\right\Vert _{\symbol{94},\frac{n-1}{2}}\leq C\left\Vert
\eta\right\Vert _{\symbol{94},\beta}.
\]

\end{lemma}

\begin{proof}
Consider the following solution for $\left(  \ref{OD}\right)  :$
\[
h\left(  r\right)  :=N_{n}\left(  r\right)  \int_{0}^{r}J_{n}\left(  s\right)
\eta\left(  s\right)  s^{n-1}ds-J_{n}\left(  r\right)  \int_{0}^{r}%
N_{n}\left(  s\right)  \eta\left(  s\right)  s^{n-1}ds.
\]
We have, for $r>1,$
\begin{align*}
\left\vert h\left(  r\right)  \right\vert  &  \leq C\left\Vert \eta\right\Vert
_{\symbol{94},\beta}r^{-\frac{n-1}{2}}\int_{0}^{r}s^{\frac{n-1}{2}-\beta}ds\\
&  \leq C\left\Vert \eta\right\Vert _{\symbol{94},\beta}\left(  1+r\right)
^{-\frac{n-1}{2}}.
\end{align*}
This gives us the desired estimate.
\end{proof}

With these a priori estimate, one could prove Theorem \ref{main1} by
contraction mapping principle, similar as before. We omit the details.

Now we proceed to prove Theorem \ref{main1'}. We consider the case $p=2.$ The
other cases are similar, but notations will be heavier. The main point to
prove this theorem is to prove a priori estimate for the solutions of the
equation $\left(  \ref{OD}\right)  .$ Since $n=3,$ the fundamental solution
$J_{3}$ and $N_{3}$ of the ODE%
\[
\varphi^{\prime\prime}+\frac{2}{r}\varphi^{\prime}+\lambda \varphi=0
\]
has the asymptotic behavior
\begin{align*}
J\left(  r\right)   &  :=J_{3}\left(  r\right)  =r^{-1}\cos\left(
\sqrt{\lambda}r-\zeta\right)  +O\left(  r^{-2}\right)  ,\text{ as }r\rightarrow+\infty.\\
N\left(  r\right)   &  :=N_{3}\left(  r\right)  =r^{-1}\sin\left(
\sqrt{\lambda}r-\zeta\right)  +O\left(  r^{-2}\right)  ,\text{ as }r\rightarrow+\infty.
\end{align*}
where $\zeta\in \mathbb{R}$ is a constant depends on $\lambda.$ Let $\rho$ be a
cutoff function such that
\[
\rho\left(  r\right)  =\left\{
\begin{array}
[c]{c}%
1,r>2,\\
0,r<1.
\end{array}
\right.
\]
The key observation is the following

\begin{lemma}
\label{lemmah}Let
\[
\eta\left(  r\right)  =\rho\left(  r\right)  r^{-2}\left[  k_{1}\sin
^{2}\left( \sqrt{\lambda} r-\zeta\right)  +k_{2}\cos^{2}\left( \sqrt{\lambda}  r-\zeta\right)  +k_{3}%
\sin\left( \sqrt{\lambda}  r-\zeta\right)  \cos\left(  \sqrt{\lambda} r-\xi\right)  \right]  +\bar{\eta}%
\]
with
\[
\left\Vert \bar{\eta}\left(  r\right)  \left(  1+r\right)  ^{3}\right\Vert
_{C^{0,\alpha}}+%
{\displaystyle\sum}
\left\vert k_{i}\right\vert <C.
\]
Then the equation
\[
h^{\prime\prime}+\frac{1}{r}h^{\prime}+\lambda h=\eta
\]
has a solution $h\in C^{0,\alpha}[0,+\infty),$ with
\[
h\left(  r\right)  =c_{1}r^{-1}\sin\left(  \sqrt{\lambda} r-\zeta\right)  +\bar{h}\left(
r\right)  ,r>1,
\]
where
\[
\left\vert \bar{h}\left(  r\right)  \right\vert +\left\vert c_{1}\right\vert
\leq C\left(  1+r\right)  ^{-2}.
\]
The solution $h$ will be denoted by $H\left(  \eta\right)  .$
\end{lemma}

\begin{proof}
Consider the solution
\begin{align*}
h\left(  r\right)   &  =N\left(  r\right)  \int_{0}^{r}J\left(  s\right)
\eta\left(  s\right)  s^{2}ds-J\left(  r\right)  \int_{0}^{r}N\left(
s\right)  \eta\left(  s\right)  s^{2}ds\\
&  =N\left(  r\right)  \int_{0}^{r}J\left(  s\right)  \rho\left(  s\right)
\sin^{2}\left(  \sqrt{\lambda} s-\zeta\right)  ds+J\left(  r\right)  \int_{r}^{+\infty
}N\left(  s\right)  \rho\left(  s\right)  \sin^{2}\left(  \sqrt{\lambda} s-\zeta\right)  ds.
\end{align*}
We have%
\[
\int_{0}^{r}J\left(  s\right)  \rho\left(  s\right)  \sin^{2}\left(
 \sqrt{\lambda}s-\zeta\right)  ds=\frac{1}{2}\int_{0}^{r}J\left(  s\right)  \rho\left(
s\right)  ds-\frac{1}{2}\int_{0}^{r}J\left(  s\right)  \rho\left(  s\right)
\cos\left(  2 \sqrt{\lambda}s-2\xi\right)  ds.
\]
Using the fact that $J\left(  s\right)  =s^{-1}\cos\left(  \sqrt{\lambda} s-\zeta\right)
+O\left(  s^{-2}\right)  ,$ we could estimate
\begin{align*}
\int_{0}^{r}J\left(  s\right)  \rho\left(  s\right)  ds  &  =\int_{0}%
^{+\infty}J\left(  s\right)  \rho\left(  s\right)  ds+O\left(  r^{-1}\right)
,\\
\int_{0}^{r}J\left(  s\right)  \rho\left(  s\right)  \cos\left(
2 \sqrt{\lambda}s-2\xi\right)  ds  &  =\int_{0}^{+\infty}J\left(  s\right)  \rho\left(
s\right)  \cos\left(  2 \sqrt{\lambda}s-2\zeta\right)  ds+O\left(  r^{-1}\right)  .
\end{align*}

Similarly, for $r>1,$
\[
\left\vert \int_{r}^{+\infty}J\left(  s\right)  \rho\left(  s\right)
\bar{\eta}\left(  s\right)  s^{2}ds\right\vert \leq\left\vert \int%
_{r}^{+\infty}s^{-1}s^{-3}s^{2}ds\right\vert \leq Cr^{-1}.
\]
Hence
\[
h\left(  r\right)  =c_{1}r^{-1}\sin\left(  \sqrt{\lambda} r-\zeta\right)  +\bar{h}\left(
r\right)  ,r>1,
\]
where
\[
\left\vert \bar{h}\left(  r\right)  \right\vert +\left\vert c_{1}\right\vert
\leq C\left(  1+r\right)  ^{-2}.
\]
This finishes the proof.
\end{proof}

\begin{definition}
The space $P_{\alpha}$ consists of functions $\eta\left(  r\right)  $ of the
form
\[
\eta\left(  r\right)  =c_{1}\rho\left(  r\right)  r^{-1}\sin\left(
 \sqrt{\lambda}r-\zeta\right)  +\bar{\eta}\left(  r\right)  ,
\]
with
\[
\left\Vert \eta\right\Vert _{\#,\alpha}=\left\vert c_{1}\right\vert
+\sup_{k\geq0}\left\Vert \left(  1+r\right)  ^{2}\bar{\eta}\left(  r\right)
\right\Vert _{C^{0,\alpha}\left(  \left[  k,k++1\right]  \right)  }.
\]

\end{definition}

With the previous results at hand, one could then use the contraction mapping
principle to prove Theorem \ref{main1'}.

\begin{proof}
[Proof of Theorem 3]The proof is similar as before. We sketch it.

We search for a solution of the form
\[
w\left(  x\right)  +Z\left(  x\right)  \left(  \varepsilon J\left(  \left\vert
y\right\vert \right)  +h\left(  r\right)  \right)  +\Phi\left(  x,y\right)  ,
\]
where $\int_{\mathbb{R}^m}%
\Phi\left(  x,y\right)  Z\left(  x\right)  dx=0,$ for any $y.$ We need to
solve
\[
\left\{
\begin{array}
[c]{l}%
-h^{\prime\prime}-\frac{1}{r}h^{\prime}-\lambda h=\int_{\mathbb{R}^m}Z\left(
x\right)  N\left(  \varepsilon J+h,\Phi\right)  dx,\\
L\Phi=N\left(  \varepsilon J+h,\Phi\right)  -Z\left(  x\right)  \int%
_{\mathbb{R}^m}Z\left(  x\right)  N\left(  \varepsilon J+h,\Phi\right)  dx.
\end{array}
\right.
\]
For each $h\in P_{0,\alpha}$ with $\left\Vert h\right\Vert _{\#,\alpha}\leq
M\varepsilon^{2}.$ We write
\[
h\left(  r\right)  =c_{1}\rho\left(  r\right)  \sin\left( \sqrt{\lambda} r-\zeta\right)
+\bar{h}\left(  r\right)  ,
\]
By Lemma \ref{S}, the second equation in this system could be solved and we
obtain a solution $\Phi=\Phi_{h}$ with
\[
\left\Vert \Phi_{h}\left(  x,y\right)  \left(  1+r\right)  ^{2}\right\Vert
_{C^{2,\alpha}}\leq C\varepsilon^{2}.
\]
We insert this solution into the first equation of the system and proceed to
solve
\begin{equation}
-h^{\prime\prime}-\frac{1}{r}h^{\prime}-\lambda h=\int_{\mathbb{R}^m}Z\left(
x\right)  N\left(  \varepsilon J+h,\Phi_{h}\right)  dx. \label{h}%
\end{equation}
Note that
\[
N\left(  \varepsilon J+h,\Phi_{h}\right)  =\left(  \varepsilon J+h\right)
^{2}+\Phi_{h}^{2}+2\left(  \varepsilon J+h\right)  \Phi_{h}=2\varepsilon
Jh+h^{2}+\varepsilon^{2}J^{2}+O\left(  \left(  1+r\right)  ^{-3}\right)  .
\]
We write the equation $\left(  \ref{h}\right)  $ as
\[
h=H\left(  \int_{\mathbb{R}^m}Z\left(  x\right)  N\left(  \varepsilon J+h,\Phi
_{h}\right)  dx\right)  :=\bar{H}\left(  h\right)  .
\]
Note that for the function $2\varepsilon Jh+h^{2}+\varepsilon^{2}J^{2},$ the
coefficient of the $r^{-2}$ part has the form
\[
k_{1}\cos^{2}\left(\sqrt{\lambda}  r-\xi\right)  +k_{2}\sin^{2}\left(\sqrt{\lambda}  r-\xi\right)
+k_{3}\sin\left(\sqrt{\lambda}  r-\xi\right)  \cos\left( \sqrt{\lambda} r-\xi\right)  .
\]
We would like to get a solution for this equation by contraction mapping
principle. By Lemma \ref{lemmah}, for each $h\in P_{\alpha}$ with $\left\Vert
h\right\Vert _{\#,\alpha}\leq M\varepsilon^{2},$ we have the estimate
\[
\left\Vert \bar{H}\left(  h\right)  \right\Vert _{\#,\alpha}\leq
C\varepsilon^{2}+CM\varepsilon^{3}.
\]
Hence for $\varepsilon$ small enough $\bar{H}$ maps the ball of radius
$M\varepsilon^{2}$ into itself and one could also check it is a contraction
map. We then get a fixed point for this map and thus complete the proof.
\end{proof}

\end{document}